\documentclass[10pt,oneside,a4paper]{amsart}
\usepackage{amsmath,amsfonts,amsthm,amssymb,textcomp}
\usepackage{array}
\usepackage{amscd}
\usepackage[all]{xy}
\usepackage{portland}

\theoremstyle{plain}
\newtheorem{theorem}{Theorem}[section]
\newtheorem{proposition}[theorem]{Proposition}
\newtheorem{lemma}[theorem]{Lemma}

\newtheorem*{question}{Question}

\newtheorem{conjecture}[theorem]{Conjecture}

\theoremstyle{definition}
\newtheorem{definition}[theorem]{Definition}

\newcommand{\C}{\mathbb{C}}

\newcommand{\Z}{\mathbb{Z}}

\renewcommand{\P}{\mathbb{P}}
\renewcommand{\L}{\ensuremath{\mathcal{L}}}
\newcommand{\M}{\ensuremath{\mathcal{M}}}
\newcommand{\N}{\ensuremath{\mathcal{N}}}
\renewcommand{\O}{\ensuremath{\mathcal{O}}}
\newcommand{\la}{\langle}
\newcommand{\ra}{\rangle}

\CompileMatrices
\SelectTips{cm}{11}

\title{Exceptional collections on some fake quadrics}
\author{Kyoung-Seog Lee}
\address{School of Mathematics, Korea Institute for Advanced Study, Seoul 130-722, Republic of Korea}
\email{kyoungseog02@gmail.com}
\author{Timofey Shabalin}
\address{National Research University Higher School of Economics, AG Laboratory, 7 Vavilova str., Moscow, Russia, 117312} 
\email{shabalin.timofey@gmail.com}
\thanks{T. S. was partially supported by AG Laboratory HSE, RF government grant, ag. 11.G34.31.0023 and RScF grant, ag. 14-21-00053}

\begin{document}
\xyoption{arc}
\maketitle

\begin{abstract}
We construct exceptional collections of maximal length on four families of surfaces of general type with $p_g=q=0$ which are isogenous to a product of curves. From these constructions we obtain new examples of quasiphantom categories as their orthogonal complements.
\end{abstract}

\section{Introduction}

Derived categories of coherent sheaves are one of the most attractive and mysterious invariants of algebraic varieties and the notion of semiorthogonal decomposition plays a key role in the study of derived categories of algebraic varieties. Semiorthogonal decompositions tell us the structure of derived categories and many interesting semiorthogonal decompositions of Fano and rational varieties were constructed. However, in contrast to the many studies of derived categories of Fano or rational varieties, we do not know much about the structure of derived categories of varieties of general type. \\
 
One of the easiest ways to construct a semiorthogonal decomposition is to find an exceptional sequence. When a triangulated category has an exceptional sequence we can divide it into the category generated by exceptional sequence and its orthogonal complement. For a surface with $p_g=q=0$, every line bundle is an exceptional object and we can construct semiorthogonal decompositions using line bundles. Then we can hope that for some surfaces with $p_g=q=0$ there are exceptional sequences of maximal lengths and we can study derived categories of these surfaces using semiorthogonal decompositions induced from them. B\"{o}hning, Graf von Bothmer and Sosna proved that there exists exceptional sequence of maximal length on the classical Godeaux surface in \cite{BBS}. They constructed the first example of a quasiphantom category as the orthogonal complement of this exceptional collection. Motivated by their results now there are lots of studies on derived categories of surfaces of general type with $p_g=q=0$. See the papers of B\"{o}hning, Graf von Bothmer, and Sosna \cite{BBS}, Alexeev and Orlov \cite{AO}, Galkin and Shinder \cite{GS}, B\"{o}hning, Graf von Bothmer, Katzarkov and Sosna \cite{BBKS}, Fakhruddin \cite{Fakhruddin}, Galkin, Katzarkov, Mellit and Shinder \cite{GKMS}, Coughlan \cite{Coughlan}, Keum \cite{Keum} and the first author \cite{lee1, lee2} for more details. They constructed categories with vanishing Hochschild homologies as orthogonal complements of exceptional sequences of line bundles of maximal lengths. Some of them are known to have finite Grothendieck groups and they provide examples of quasiphantom categories. Supported by these examples, it seems that the following question is now considered by many experts.

\begin{question}
Let $S$ be a smooth projective surface of general type with $p_g=q=0$. Is there an exceptional sequence whose length is equal to the rank of Grothendieck group of $S$ or the total dimension of $H^*(S,\mathbb{C})$? Especially can we construct such an exceptional sequence using line bundles on $S$?
\end{question}
 
We want to answer the above question for some special surfaces of general type with $p_g=q=0$. Bauer, Catanese and Grunewald have classified surfaces of general type with $p_g=q=0$ which are quotients of a product of curves by the free diagonal action of a finite group in \cite{bcg}. There are 12 families of such surfaces and these are called the surfaces isogenous to a higher product of unmixed type. The rank of Grothendieck group of every such surface is 4 and the total dimension of cohomology group of every such surface is also 4 \cite{GS}. Therefore the maximal possible length of the exceptional sequence on every such surface is 4. For the 4 families of such surfaces with abelian group quotients, exceptional collections of maximal length were constructed in \cite{GS, lee1, lee2}. In this paper we construct such collections in four more cases where $G$ is $D_4 \times \Z/2$, $S_4$, $S_4 \times \Z/2$ and $(\Z/4 \times \Z/2) \rtimes \Z/2$ ($G(16)$ in the notation of \cite{bcg}).

\begin{theorem}
Let $S=(C \times D)/G$ be a surface isogenous to a higher product with $p_g=q=0$, where $G$ is one of $D_4 \times \Z/2$, $S_4$, $S_4 \times \Z/2$ and $(\Z/4 \times \Z/2) \rtimes \Z/2$. Then there are exceptional sequences of line bundles of maximal length 4 on $S$ and the orthogonal complements of these exceptional sequences are quasiphantom categories.
\end{theorem}

We think that we can generalize this result to any surface isogenous to a higher product with $p_g=q=0$. The following conjecture has also appeared in \cite{GS}.

\begin{conjecture}
Let $S$ be a  surface isogenous to a higher product with $p_g=q=0$. Then there are exceptional sequences of line bundles of maximal length 4 on $S$.
\end{conjecture}

We recall some basic facts about these surfaces (see \cite{bauercat, Pardini} for more details) and sketch the idea of the construction. Let $S = (C_1 \times C_2) / G$ be one of them. We have $C_i / G \cong \P^1$, $|G| = (g_1 - 1)(g_2 - 1)$ where $g_i$ is the genus of $C_i$. Since $p_g = q = 0$, the Chern class map $Pic(S) \to H^2(S,\Z)$ is an isomorphism. It follows from the Noether's formula that $H^2(S,\Z)$ has rank 2. Thus up to a finite torsion subgroup $Pic(S)$ is an unimodular indefinite lattice of rank 2, that is a hyperbolic plane. Let $p_i : C_1 \times C_2 \to C_i$ be the projections and denote by $\mathcal{F} \boxtimes \mathcal{G} = p_1^{\ast}(\mathcal{F}) \otimes p_2^{\ast}(\mathcal{G})$ the external tensor product of coherent sheaves $\mathcal{F}$ and $\mathcal{G}$ on $C_1$ and $C_2$. Let us denote by $\O(2,0)$ and $\O(0,2)$ the classes of $p^G_{\ast}p_1^{\ast}(\Omega^1_{C_1})$ and $p^G_{\ast}p_2^{\ast}(\Omega^1_{C_2})$ in the lattice $H^2(S,\Z) / Tors$. Then we have $\O(2,0)^2 = 0$, $\O(0,2)^2 = 0$, $\O(2,0) \cdot \O(0,2) = 4$. We see that the lattice $H^2(S,\Z) / Tors$ must be generated by some numerical halves $\O(1,0)$ and $\O(0,1)$ of canonical classes of curves $C_1$, $C_2$. The Euler characteristic of a line bundle on $S$ of numerical type $\O(i,j)$ is $(i-1)(j-1)$. 

The category $\text{coh}(S)$ of coherent sheaves on $S$ is equivalent to the category of $G$-equivariant coherent sheaves on $C_1 \times C_2$ and we denote the functor 
$$\text{coh}^G(C_1 \times C_2) \to \text{coh}(S)$$
by $p^G_{\ast}$. Therefore we are going to construct exceptional sequences of line bundles in $D^b(S)$ by constructing exceptional sequences of $G$-equivariant line bundles in $D^b_G(C_1 \times C_2)$. Recall the definition of exceptional sequence.

\begin{definition}
(1) An object $E$ of a triangulated category $D$ is called exceptional if
$$
Hom^k(E, E) = \left\{
\begin{array}{l l}
\C, & \text{ if } k = 0, \\
0, & otherwise
\end{array}
\right.
$$
(2) A sequence of exceptional objects $E_1, \dots, E_n$ is called exceptional if
$$
Hom^k(E_i, E_j) = 0
$$
for all $i > j$ and all $k$.
\end{definition}

From the definition it is clear that when $\L, \O$ is an exceptional sequence then $\chi(\L)$ should be $0$. Therefore we need some numerical halves $\O(1,0)$ and $\O(0,1)$ of canonical classes of curves $C_1$, $C_2$ to construct exceptional sequence of line bundles. However there are some cases when we cannot give a $G$-equivariant structure to the numerical halves of canonical bundles. In these cases we construct equivariant bundles on $C_1 \times C_2$ by finding two divisors $D_1$ and $D_2$ on $C_1$ and $C_2$ such that each of them is not equivariant on the curve $C_i$ but they have inverse obstructions and therefore $p_1^{\ast}\O(D_1) \otimes p_2^{\ast}\O(D_2)$ is equivariant on the product. From now on we will omit $p_{\ast}^G$, $p_1^{\ast}$ and $p_2^{\ast}$ from our notation.

Let us explain how this is possible. For any variety $X$ with an action of a finite group $G$ there is an exact sequence
\begin{equation}\label{exseq}
0 \to \widehat G \to Pic^G(X) \to Pic(X)^G \to H^2(G, \C^\times),
\end{equation}
where $\widehat G = Hom(G, \C^\times)$ is the group of characters of $G$, $Pic^G(X)$ is the group of $G$-equivariant line bundles on $X$ and $Pic(X)^G$ is the group of line bundles whose classes in the Picard group are invariant under the action of $G$. The last map in \eqref{exseq} providing the obstruction to the existence of an equivariant structure on a line bundle $\L$ may be described as follows. Fix $\L$ in $Pic(X)^G$. For each $g \in G$ pick some isomorphisms $\phi_g : g_{\ast} \L \to \L$. Then 
\begin{equation}\label{obstr}
\eta_{\L}(g,h) = \phi_g \cdot g_{\ast} \phi_h \cdot \phi_{gh}^{-1}
\end{equation}
is an automorphism of $\L$, that is, an element of $\C^{\times}$. Therefore $\eta_{\L}$ is a 2-cocycle in $Z^2(G,\C^{\times})$ and the map $Pic(X)^G \to H^2(G, \C^\times)$ is given by $\L \mapsto \eta_{\L}$.

Let $S=(C_1 \times C_2)/G$ be surface isogenous to a higher product of unmixed type with $p_g=q=0$ and $g_1 \leq g_2$. For $G=D_4 \times \Z/2, S_4, S_4 \times \Z/2$ cases we cannot give $G$-equivariant structure to any half of canonical line bundle on $C_1$. What we can do is to construct $G$-invariant acyclic line bundle $\L$ on $C_1$ which is a numerical half of the canonical line bundle but not $G$-equivariant. Then we need to find a $G$-invariant line bundle $\M$ such that $\eta_{\L} \cdot \eta_{\M} = 0$ in $H^2(G, \C^\times)$ to make $\L \boxtimes \M$ a $G$-equivariant acyclic line bundle on $C_1 \times C_2$. The following proposition of Dolgachev \cite{dolgachev} tells us that we can always find such a line bundle.

\begin{proposition} \cite[Proposition 2.2]{dolgachev} \label{dolgachevproposition}
If $X$ is a curve, then the map $Pic(X)^G \to H^2(G, \C^\times)$ is surjective.
\end{proposition}

In fact there are infinitely many such line bundles. Then we show that there are $G$-equivariant acyclic line bundles on $C_2$. From Serre duality, K\"{u}nneth formula and the Riemann-Roch theorem on the curves $C_1$, $C_2$ one obtains the following lemma. 

\begin{lemma} \label{collection}
Suppose that $\L$ is a $G$-invariant acyclic line bundle on $C_1$, $\M$ and $\N$ line bundles on $C_2$ such that $\M$ is $G$-invariant and $\eta_L \cdot \eta_M = 0$ in $H^2(G, \C^\times)$ and $\N$ is acyclic and admits $G$-equivariant structure on $C_2$. Then the sequence
$$
\L \boxtimes (\M \otimes \N)(\chi_1),\; \L \boxtimes \M(\chi_2),\; \O \boxtimes \N(\chi_3),\; \O
$$
is an exceptional collection on $S$. Here $\chi_i$ denote arbitrary characters of $G$ by which we can twist the equivariant structure.
\end{lemma}

We will construct exceptional sequence of maximal length on $S=(C_1 \times C_2)/G$ by this method when $G=D_4 \times \Z/2, S_4, S_4 \times \Z/2$. When $G=G(16)$ then we can find acyclic $G$-equivariant line bundles on $C_1, C_2$ and the construction becomes much easier. \\

\noindent \textbf{Acknowledgement.} 
The authors of this paper met in the GRIFGA/Lebesgue summer school --- Derived Categories which was held from June 23rd to 27th in the University of Nantes. We thank organizers of the summer school for their kind hospitality. The first author is grateful to his advisor Young-Hoon Kiem for his encouragement and many suggestions for the first draft of this paper. He thanks Sergey Galkin, Ludmil Katzarkov, Minhyong Kim, Yongnam Lee, Miles Reid for helpful conversations and encouragement. He thanks Fabrizio Catanese for answering his questions and helpful conversations. He also thanks Seoul National University for its support during the preparation of this paper. The second author would like to thank his advisor Dmitri Orlov for his attention and encouragement and also Sergey Galkin and Yakov Kononov for helpful discussions.

\section{Invariant line bundles}

In this section we recall some results of Beauville \cite{beauville} about curves with $G$-action and invariant line bundles on them which will be extremely useful for our construction. Let $C$ be a curve and $G$ be a group acting on $C$. Let $B$ be the quotient curve, $\pi: C \to B$ the quotient map. Denote by $R_C$ the field of rational functions on the curve $C$. From the short exact sequence of $G$-modules $$ 0 \to \mathbb{C}^* \to R_C^* \to R_C^*/\mathbb{C}^* \to 0, $$ we get $$ 0 \to \mathbb{C}^* \to R_B^* \to (R_C^*/\mathbb{C}^*)^G \to H^1(G,\mathbb{C}^*) \to 0, $$
since $H^1(G,R_C^*)=0 $ by Hilbert's Theorem 90. From the short exact sequence of $G$-modules $$ 0 \to R_C^*/\mathbb{C}^* \to Div(C) \to Pic(C) \to 0, $$ we get the following commutative diagram
$$ \begin{CD}
0 @>>> R^*_B/{\mathbb{C}^*} @>>> Div(B) @>>> Pic(B) @>>> 0 \\
@. @VVV @VVV @VVV @. \\
0 @>>> (R^*_{C}/{\mathbb{C}^*})^G @>>> Div(C)^G @>>> Pic(C)^G @>>> H^1(G,R^*_{C}/{\mathbb{C}^*}).
\end{CD} $$
If we change the lower exact sequence by $$ 0 \to (R_C^*/\mathbb{C}^*)^G \to Div(C)^G \to Im \to 0, $$ we still have a commutative diagram and we can apply the snake lemma as follows.
\begin{equation} \label{diagram}
\begin{CD}
0 @>>> 0 @>>> 0 @>>> X @>>> @. \\
@. @VVV @VVV @VVV @. \\
0 @>>> R^*_{B}/{\mathbb{C}^*} @>>> Div(B) @>>> Pic(B) @>>> 0 \\
@. @VVV @VVV @VVV @. \\
0 @>>> (R^*_{C}/{\mathbb{C}^*})^G @>>> Div(C)^G @>>> Im @>>> 0 \\
@. @VVV @VVV @VVV @. \\
 @>>> H^1(G,\mathbb{C}^*) @>>> Y @>>> Z @>>> 0. \\
\end{CD}
\end{equation}
Sometimes we can compute $X, Y, Z$ explicitly and then the above diagram becomes very useful. For example when $B$ is isomorphic to $\mathbb{P}^1$ then we get $X=0$.

\begin{lemma}{\cite{beauville}}
Let $C$ be a curve with involution $\sigma$, $B$ be the quotient curve $C/\langle \sigma \rangle$. If the covering $\pi: C \to B$ is unramified, then $Y = Z = 0$ and $X \cong H^1(G, \C^\times) \cong \Z/2$. If the set $R$ of ramification points of $\pi$ is non-empty, then $X = 0$ and the last row is isomorphic to
$$ 0 \to \mathbb{Z}/2 \to (\mathbb{Z}/2)^R \to Pic(C)^{\sigma}/\pi^*Pic(B) \to 0, $$
where the kernel is generated by $(1,\dots,1)$.
\end{lemma}

The next result of Beauville will be very important for our computations.

\begin{lemma}{\cite{beauville}}\label{beauvillelemma}
In the situation of the previous lemma, let $\L$ be a $\sigma$-invariant theta characteristic on $C$. There are some $\L' \in Pic(B)$ and $E \subset R$ such that $\L = \pi^*(\L')(E)$. Then we also have $\L = \pi^*(K_B \otimes \L'^{-1})(R - E)$. The pushforward of $\L$ splits: $\pi_*(\L) = \L' \oplus (K_B \otimes \L'^{-1})$ and we have $$ H^0(C, \L) = H^0(B, \L') \oplus H^1(B, \L')^*. $$
\end{lemma}

\section{Case $G = D_4 \times \Z/2$} \label{sectiond4}

The group $G$ has a presentation 
$$
\langle x,y,z \mid x^4 = y^2 = z^2 = [x,z] = [y,z] = 1, x^y = x^{-1} \rangle,
$$
where $x^y = y^{-1}xy$, $[x,y] = xyx^{-1}y^{-1}$.

A covering $\pi: C \to \P^1$ with Galois group $G$ can be specified by its ramification type $(m_1^{k_1}, \dots, m_l^{k_l})$ which means that $\pi$ has $k_i$ ramification points of multiplicity $m_i$ and by the tuple of generators $(g_1, \dots, g_n)$, $g_i \in G$, $n = k_1 + \dots + k_l$ such that a simple geometric loop around $j$-th ramification point on $\P^1$ lifts to the action of $g_j$ on $C$. We must have $g_1 \dots g_n = 1$ and $g_1, \dots, g_n$ must generate $G$. Of course these data do not specify the covering completely because one can move the ramification points on $\P^1$. If $p_1, \dots, p_n \in \P^1$ are the ramification points then we will denote by $E_i$ the reduced fiber of $\pi$ over $p_i$.

The covering $C_1 \to \P^1$ has ramification type $(2^3, 4)$ and the corresponding elements of $G$ are $(z, yz, xy, x)$. The covering $C_2 \to \P^1$ has ramification type $(2^6)$ and the corresponding tuple is $(y, x^3yz, x^2y, x^3yz, x^2z, x^2z)$. The curve $C_1$ has genus 3 ($2g-2 = 4$), $C_2$ has genus 9 ($2g-2 = 16$). Divisors $E_1,E_2,E_3$ on $C_1$ have degree 8 and $E_4$ has degree 4. All divisors $E_i$ on $C_2$ consist of 8 points. 

\begin{lemma}
There is a $G$-invariant theta characteristic $\L$ on $C_1$ which has no sections.
\end{lemma}
\begin{proof}
Consider the mapping $\pi: C_1 \to C_1 / \la z \ra$. The quotient has genus 0, so $\pi$ is a hyperelliptic structure on $C_1$. It is ramified in the 8 points of $E_1$. The quotient of $C_1$ by subgroup $\la x^2, xy \ra$ also has genus 0. The divisor $E_1$ consists of two $\la x^2, xy \ra$-orbits. Let $B_1$ be one of them. Let also $B_2$ be any full fiber of $\pi$. Since the subgroups $\la x^2, xy \ra$ and $\la z \ra$ are normal in $G$, divisors $B_1$ and $B_2$ have $G$-invariant classes in the Picard group. Let $L = B_1 - B_2$ and $\L = \O(L)$. The canonical class of $C_1$ is equivalent to $E_1 - 2B_2 \sim 2B_1 - 2B_2$, so $L$ is a theta characteristic. By Lemma \ref{beauvillelemma} applied to $\pi$ we have $h^0(C_1, L) = 2h^0(\P^1, \O(-1)) = 0$.
\end{proof}

The proposition \ref{dolgachevproposition} implies the next lemma.

\begin{lemma}
There is a $G$-invariant line bundle $\M$ on $C_2$ such that $\eta_\L \cdot \eta_\M = 0$.
\end{lemma}

Now we construct an explicit $G$-equivariant acyclic theta characteristic on $C_2$.

\begin{lemma}
There is a $G$-equivariant theta characteristic $\N$ on $C_2$ which has no sections.
\end{lemma}
\begin{proof}
Let $N = E_1 - E_2 + E_5$ and $\N = \O(N)$. Quotients of $C_2$ by subgroups $\la xyz, x^2, z \ra, \la y, x^2, z \ra, \la y, xyz, x^2 \ra$ all have genus 0. From these three quotients we see that $E_1 \sim E_3$, $E_2 \sim E_4$, $E_5 \sim E_6$. It follows that $\N$ is a theta characteristic. Consider the quotient $\pi: C_2 \to C_2 / \la x^2z \ra$. We have $N = \pi^* N' + E_5$, where $N'$ is a divisor of degree 0. The curve $C_2 / \la x^2z \ra$ has genus 1. We have $2N' \sim 0$ which is the canonical class of $C_2 / \la x^2z \ra$. There is an induced action of $y$ on $C_2 / \la x^2z \ra$. Applying Lemma \ref{beauvillelemma} first to $\pi$, then to the quotient of $C_2 / \la x^2z \ra$ by $\la y \ra$ we find $h^0(C_2, N) = 2h^0(C_2 / \la x^2z \ra, N') = 4h^0(\P^1, \O(-1)) = 0$.
\end{proof}

From the above lemmas we get the following theorem.

\begin{theorem}
Let $S=(C_1 \times C_2)/G$ be a surface isogenous to a higher product with $p_g=q=0$ and $G = D_4 \times \Z/2$. Then there are exceptional sequences of line bundles of maximal length 4 on $S$.
\end{theorem}

\section{Case $G = S_4 \times \Z/2$}

The covering $C_1 \to \P^1$ has ramification type $(2,4,6)$. The corresonding tuple is $((12), 0), ((1234), 1), ((432), 1)$. The covering $C_2 \to \P^1$ has ramification type $(2^6)$ and the tuple is $((12)(34), 1), ((12), 1), ((34), 1), ((14)(23), 1), ((23), 1), ((14), 1)$. 

\begin{lemma}\label{bundle}
There is a $G$-invariant acyclic theta characteristic $\L$ on $C_1$.
\end{lemma}
\begin{proof}
Note that the curve $C_1$ is hyperelliptic, the quotient of $C_1$ by the action of the element $(1,1)$ of order two in $G$ is isomorphic to $\P^1$. We denote by $\pi: C_1 \to \P^1$ the quotient morphism. Consider the Klein subgroup $V_4 = V_4 \times \{ 0 \} \le S_4 \times \Z_2$. The quotient $C_1 / V_4$ has genus 0. The group $V_4$ acts freely on $E_3$, so $E_3$ consists of two free $V_4$-orbits, which are equivalent to each other. Let $W$ be one of them. Then $\O(W)$ is a $G$-invariant line bundle. Let $\L = \pi^*\O(-1) \otimes \O(W)$. Note that the ramification points of $\pi$ are precisely $E_3$. Looking at the morphism $\pi$ we see that $\pi^*\O(-2)(E_3)$ is a canonical bundle on $C_1$, but $E_3 \sim 2W$, so $\L$ is a theta characteristic on $C_1$. It is $G$-invariant since $G$ must preserve the hyperelliptic structure. From Lemma \ref{beauvillelemma} applied to the hyperelliptic involution we get $H^0(C_1,\L) = H^0(\P^1, \O(-1)) \oplus H^1(\P^1, \O(-1)) = 0$.
\end{proof}

The next lemma follows from proposition \ref{dolgachevproposition}.

\begin{lemma}
There is a $G$-invariant line bundle $\M$ on $C_2$ such that $\eta_\L \cdot \eta_\M = 0$.
\end{lemma}

Then we construct an explicit $G$-equivariant acyclic theta characteristic on $C_2$.

\begin{lemma}
There is a $G$-equivariant acyclic theta characteristic $\N$ on $C_2$.
\end{lemma}
\begin{proof}
Note that the elements $((12)(34), 1)$, $((13)(24), 1)$, $((14)(23), 1)$ generate the subgroup $V_4 \times \Z_2$ isomorphic to $\Z_2^3$. Consider the system of quotients $C_2 \to C_2' \to C_2'' \to C_2'''$, where $C_2'$ is the quotient by the action of $((12)(34), 1)$, $C_2''$ by the induced action of $((13)(24), 1)$ and $C_2'''$ by $((14)(23), 1)$. We divide $E_1$ into three parts $F_1, F_2, F_3$ according to their stabilizers $((12)(34), 1)$, $((13)(24), 1)$ or $((14)(23), 1)$ and analogously $E_4$ into $F_4, F_5, F_6$. The ramification points of the morphism $C_2 \to C_2'$ are $F_1 \cup F_4$, the ramification of $C_2' \to C_2''$ is $F_2' \cup F_5'$, where $F_i'$ is the image of $F_i$ on $C_2'$ (without multiplicities) and the ramification of $C_2'' \to C_2'''$ is $F_3'' \cup F_6''$. The curve $C_2'''$ is elliptic, so its canonical divisor $K_{C_2'''}$ is zero. Thus we see that $K_{C_2''} = F_3'' + F_6''$, $K_{C_2'} = F_2' + F_3' + F_5' + F_6'$ and $K_{C_2} = F_1 + \dots + F_6 = E_1 + E_4$.

The element $((14), 1)$ stabilizes the subgroup generated by $((12)(34), 1)$, $((13)(24), 1)$. Therefore it acts on the curve $C_2''$. The quotient by this action is $\P^1$ and $F_3'', F_6''$ are its free orbits. Therefore $F_3'' \sim F_6''$ and $F_3' \sim F_6'$, $F_3 \sim F_6$. Analogously $F_2' \sim F_5'$ and so on. 

Let $N = E_1 + E_2 - E_3$, $N' = F_2' + F_3' + E_2' - E_3'$, $N'' = F_3'' + E_2'' - E_3''$. From the natural map $C_2''' \to C_2 / G$ we get $2E_2''' \sim 2E_3'''$. Therefore $N, N', N''$ are theta characteristics on the corresponding curves. From the Beauville's lemma applied repeatedly we get $h^0(C_2, N) = 2h^0(C_2', N') = 4h^0(C_2'', N'')$.

Now consider again the quotient $C_2'' \to \P^1$ by the action of $((14), 1)$. Applying Beauville's lemma to it we find $h^0(C_2'', N'') = 2h^0(\P^1, \O(-1)) = 0$.
\end{proof}

From the above lemmas we get the following theorem.

\begin{theorem}
Let $S=(C_1 \times C_2)/G$ be a surface isogenous to a higher product with $p_g=q=0$ and $G = S_4 \times \Z/2$. Then there are exceptional sequences of line bundles of maximal length 4 on $S$.
\end{theorem}

\section{Case $G = S_4$}

The covering $C_1 \to \P^1$ has ramification type $(3,4^2)$. The corresponding tuple is $((123), (1234), (1243))$. The covering $C_2 \to \P^1$ has ramification type $(2^6)$ and the tuple is $((12), (12), (23), (23), (34), (34))$. The following lemma was stated in \cite{Pardini}. We give a proof as follows.

\begin{lemma} Let us denote by $C_1'$ the curve $C_1$ from the previous section with its $S_4 \times \Z_2$-action. Consider the subgroup $S_4$ embedded into $S_4 \times \Z_2$ by the mapping $(id, sign)$. We claim that the curve $C_1'$ with the action of this subgroup is isomorphic to $C_1$ with its $S_4$-action.
\end{lemma}
\begin{proof}
Consider the quotient $C_1' \to C_1'/S_4$ and its ramification. There is a map $C_1'/S_4 \to C_1'/(S_4 \times \Z_2)$ which is a twofold covering of $\P^1$ by $\P^1$ and we want to choose coordinates in such a way that the mapping will be given by $z \mapsto z^2$. Suppose that the covering $C_1' \to C_1'/(S_4 \times \Z_2) \cong \P^1$ is ramified over points $0,1,\infty$ and the loop around $0$ corresponds to the action of $((12), 0)$ on the covering and the loops around $1$ and $\infty$ correspond to $((1234), 1)$, $((432), 1)$. Choose $-1$ as a base point on the quotient. The covering $C_1'$ is specified by the map
$$
\pi_1(\P^1 \setminus \{ 0, 1, \infty \}, -1) = \la t_0, t_1, t_\infty \mid t_0t_1t_\infty = 1 \ra \to \Z_2 \times S_4,
$$
given by
$$
t_0 \mapsto ((12), 0), \; t_1 \mapsto ((1234), 1), \; t_\infty \mapsto ((432), 1).
$$

There are four points $0, 1, -1, \infty$ on $C_1'/S_4$ above $0,1,\infty$ on $C_1'/(S_4 \times \Z_2)$. We want to compute the composite map
$$
\pi_1(\P^1 \setminus \{ 0, \pm 1, \infty \}, i) \to \pi_1(\P^1 \setminus \{ 0, 1, \infty \}, -1) \to S_4 \times \Z_2,
$$
check that its image lies in $S_4 \le S_4 \times \Z_2$, that loop around one of the points maps to the trivial element, so the ramification is actually in the three points, and finally check that the map gives our covering $C_1$. 

If we choose the generators of
$$
\pi_1(\P^1 \setminus \{ 0, \pm 1, \infty \}, i) = \la s_1,s_{-1},s_0,s_\infty \mid s_1 s_{-1} s_0 s_\infty = 1 \ra
$$
as follows
$$
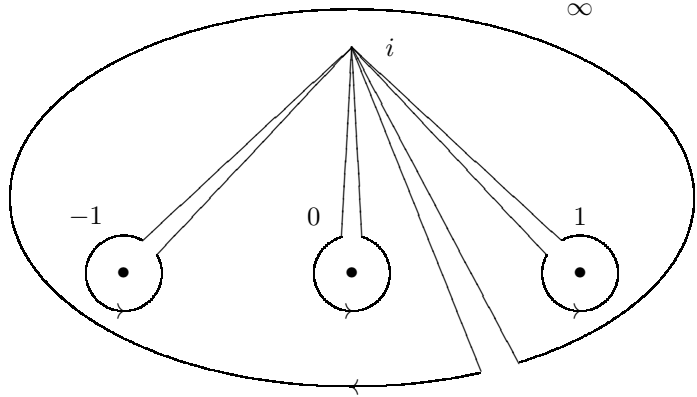
\xy
<1cm,0cm>:,
(0.5,0)*{i}, (-3.5,-2.25)*{-1}, (-0.5,-2.25)*{0}, (3,-2.25)*{1}, (3,0.5)*{\infty},
(-3,-3)*=0{\bullet}, {\ellipse(0.5)u:a(60),u:a(30){-}},
(0,-3)*=0{\bullet}, {\ellipse(0.5)u:a(105),u:a(75){-}},
(3,-3)*=0{\bullet}, {\ellipse(0.5)u:a(150),u:a(120){-}},
0="z";(-3,-3)+/r:a(30).5cm/, **@{-},
0;(-3,-3)+/r:a(60).5cm/, **@{-},
0;(0,-3)+/r:a(75).5cm/, **@{-},
0;(0,-3)+/r:a(105).5cm/, **@{-},
0;(3,-3)+/r:a(120).5cm/, **@{-},
0;(3,-3)+/r:a(150).5cm/, **@{-},
0;(1.7,-4.3),**@{-},
0;(2.2,-4.2),**@{-},
(0,-2), {\ellipse(4.5,2.5)u:a(-72.5),u:a(-77.5){-}},
(-3,-3.5)*r\dir{>},
(0,-3.5)*r\dir{>},
(3,-3.5)*r\dir{>},
(0,-4.5)*r\dir{<},
(0,-4.75)*@{}, (-4.75,0)*@{}
\endxy
$$
(these are loops around $1, -1, 0, \infty$) and generators $t_0, t_1, t_\infty$ as follows,
$$
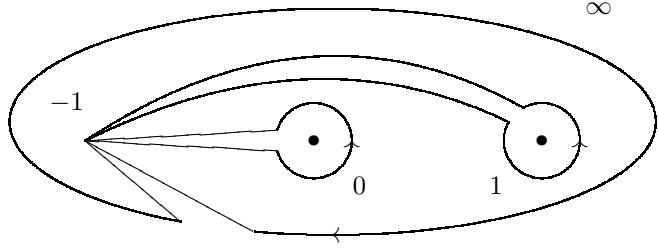
\xy
<1cm,0cm>:,
(-3.25,0.5)*{-1}, (0.6,-0.6)*{0}, (2.4,-0.6)*{1}, (3.75,1.75)*{\infty},
0*=0{\bullet}, {\ellipse(0.5)u:a(-15),_,u:a(15){-}}, (0.5,0)*u\dir{>},
(3,0)*=0{\bullet}, {\ellipse(0.5)u:a(-60),_,u:a(-30){-}}, (3.5,0)*u\dir{>},
(-3,0);/r:a(165).5cm/,**@{-},
(-3,0);/r:a(195).5cm/,**@{-},
(-3,0);(3,0)+/ul:a(-15).5cm/, **\crv{(0,2)},
(-3,0);(3,0)+/ul:a(15).5cm/, **\crv{(0,1.5)},
(0.25,0.25), {\ellipse(4.25,1.5)u:a(-95),u:a(-100){-}},
(0.25,-1.25)*r\dir{<},
(-3,0);(-0.8,-1.2), **@{-},
(-3,0);(-1.75,-1.075), **@{-},
(0,2)*@{}, (0,-2)*@{}, (-4.25,0)*@{}
\endxy
$$
then we can just draw the images of $s_1, \dots, s_\infty$ under the map $z \mapsto z^2$ and then write them as combinations of $t_0, t_1$. We obtain that the homomorpism 
$$
\pi_1(\P^1 \setminus \{ 0, \pm 1, \infty \}, i) \to \pi_1(\P^1 \setminus \{ 0, 1, \infty \}, -1)
$$
is given by
$$
s_1 \mapsto t_1,\; s_{-1} \mapsto t_0t_1t_0^{-1},\; s_0 \mapsto t_0^2,\; s_\infty \mapsto t_0^{-1}t_1^{-1}t_0^{-1}t_1^{-1}.
$$
So the composition
$$
\pi_1(\P^1 \setminus \{ 0, \pm 1, \infty \}, i) \to \pi_1(\P^1 \setminus \{ 0, 1, \infty \}, -1) \to S_4 \times \Z_2
$$
is given by
$$
s_1 \mapsto ((1234), 1),\; s_{-1} \mapsto ((1342), 1),\; s_0 \mapsto (1, 0),\; s_\infty \mapsto ((134), 0).
$$
We see that it factors through $S_4 \le S_4 \times \Z_2$ and that $s_0$ is mapped to trivial element, so $C_1' \to C_1'/S_4$ is ramified only over $\pm 1, \infty$. The group $\pi_1(\P^1 \setminus \{ \pm 1, \infty\})$ can be represented as
$$
\la s_\infty, s_1, s_{-1} \mid s_\infty s_1 s_{-1} = 1 \ra
$$
and these three generators map to $(134), (1234), (1342)$. If we conjugate this sequence by $(13)(24)$, we get $(123),(1234),(1243)$. Therefore, the covering $C_1' \to C_1'/S_4$ is isomorphic to $C_1 \to C_1/S_4$ with $S_4$-action.
\end{proof}

\begin{lemma}
There is a $G$-invariant acyclic theta characteristic $\L$ on $C_1$.
\end{lemma}
\begin{proof}
Indeed, we can use the same $\L$ as in Lemma \ref{bundle}.
\end{proof}

The proposition \ref{dolgachevproposition} implies the next lemma.

\begin{lemma}
There is a $G$-invariant line bundle $\M$ on $C_2$ such that $\eta_\L \cdot \eta_\M = 0$.
\end{lemma}

Then we prove that there is a $G$-equivariant acyclic theta characteristic on $C_2$.

\begin{lemma}
There is a $G$-equivariant acyclic theta characteristic $\N$ on $C_2$.
\end{lemma}
\begin{proof}
We let $N = E_i + E_j - E_k$ for $i,j,k$ all different and $\N = \O(N)$. We will prove that for some choices of $i,j,k$ such $\N$ is acyclic, but we can't say for which ones precisely. 

Note that $N$ has degree 12 so we only have to check that $N$ has no regular sections. The subgroup $A_4$ acts freely on $C_2$, the quotient $C_2/A_4$ has genus 2, the morphism $C_2/A_4 \to C_2/S_4$ has 6 ramification points $E'_1, \dots, E_6'$, where $E_i'$ is the image of $E_i$ on the quotient $C_2/A_4$. We have $E_1 + E_2 + E_3 \sim E_4 + E_5 + E_6$ and the $S_4$ acts by a sign character on the function with divisor $E_1 + E_2 + E_3 - E_4 - E_5 - E_6$. All relations between divisors $E_i$ follow form this one and $2E_i \sim 2E_j$ since $G$ must act via character on a function giving such relation. Now it is not hard to see that possible choices of $i,j,k$ give 10 different classes in $Pic(C_2)$. It follows from these relations that $\N$ is a theta characteristic. \\ \\
Step 1. There are no 1-dimensional subrepresentations in $H^0(C_2, N)$. \\
The subgroup $A_4$ must act trivially on such a subrepresentation, so if it exists then we must have $E_i + E_j \sim E_k + E_l$. This never holds when $i,j,k$ are different. \\ \\
Step 2. The dimension of $H^0(C_2, N)$ is 0, 2 or 4. \\
We apply the Beauville's lemma to the quotient $C_2 \to C_2 / (12)$. We have $h^0(C_2, N) = h^0(A) + h^1(A) = 2h^0(A)$ where $A$ is a line bundle of degree 3 on the quotient $C_2 / (12)$. By Clifford's theorem $h^0(A) \le 2$. \\ \\
Step 3. There are $i,j,k$ such that $N$ has no regular sections. \\
From previous steps we know that $H^0(C_2, N)$ is a sum of 2-dimensional irreducible representations of $S_4$. Let $W$ be one of them. Then $W$ comes from an irreducible representation of $S_3$ via the map $S_4 \to S_4/V_4 \cong S_3$. There is a basis $\phi_1, \phi_2$ of $W$ such that $A_3 \le S_3$ acts on each $\phi_i$ by a character, and elements not in $A_3$ exchange $\phi_1$ and $\phi_2$, also multiplying them by some constant. We have $E_1 + E_2 - E_3 + (\phi_1) = F_1$ and $E_1 + E_2 - E_3 + (\phi_2) = F_2$ where $F_1$ and $F_2$ are $A_4$-invariant divisors of degree 12. Elements not in $A_4$ exchange $F_1$ and $F_2$. We denote by $E_i'$ and $F_i'$ images of $E_i$ and $F_i$ on the quotient $C_2/A_4$ (without multiplicities). Then $E_i'$ and $F_i'$ are single points on $C_2/A_4$. Note that $F_1', F_2'$ are not equal to any $E_i'$. We have $\sigma(F_1') = F_2'$, where $\sigma$ is the nontrivial automorphism of the covering $C_2/A_4 \to C_2/S_4$.

We want to prove that there no more than 2 different $N$'s of the form $E_i + E_j - E_k$ which have sections. Since the morphism $C_2 \to C_2/A_4$ is unramified, from the diagram \ref{diagram} we see that the kernel $X$ of the map $Pic(C_2/A_4) \to Pic(C_2)$ is isomorphic to $H^1(A_4, \C^\times) \cong \Z/3$. Suppose that $N(i,j,k) = E_i + E_j - E_k$ and different $N(l,m,n)$ both have sections and we have $N(i,j,k) \sim F_1 \sim F_2$, $N(l,m,n) \sim F_3 \sim F_4$, where $\sigma(F_1) = F_2$, $\sigma(F_3) = F_4$, each $F_i$ is a pullback of a point on $C_2/A_4$. Note that $F_1'$ is not equal to $F_3'$ or $F_4'$, as otherwise we would have $N(i,j,k) \sim N(l,m,n)$. Divisors $F_1' - F_2'$ and $F_3' - F_4'$ are nontrivial elements of $X$. Since $X \cong \Z/3$, we have either $F_1' - F_2' \sim F_3' - F_4'$ or $2(F_1' - F_2') \sim F_3' - F_4'$. The first variant is impossible, because then the divisor $F_1' + F_4'$ has at least 2 dimensional sections, but such a divisor is unique, so it must be equivalent to $F_1' + F_2'$, which is not true. There can't be a third pair $F_5', F_6'$ because then $F_5' - F_6'$ must be equivalent to one of $F_1' - F_2', F_3' - F_4'$. Therefore, there are not more than 2 $N$'s which have sections, so acyclic $N$ exists.
\end{proof}

From the above lemmas we get the following theorem.

\begin{theorem}
Let $S=(C_1 \times C_2)/G$ be a surface isogenous to a higher product with $p_g=q=0$ and $G = S_4$. Then there are exceptional sequences of line bundles of maximal length 4 on $S$.
\end{theorem}

\section{Case $G = (\Z/4 \times \Z/2) \rtimes \Z/2$}

The group $G$ has a presentation
$$
\la x, y, z \mid x^4 = y^2 = z^2 = [x,y] = [y,z] = 1, x^z = xy \ra.
$$
Both coverings $C_1,C_2 \to \P^1$ have ramification type $(2^2,4^2)$. The corresponding tuples are $(z,z,x,x^{-1})$ for $C_1$ and $(x^2yz, x^2yz, xyz, x^3z)$ for $C_2$. Both curves $C_1$ and $C_2$ have genus $5$ ($2g-2 = 8$). The reduced fibers $E_1$, $E_2$ consist of 8 points each and $E_3$, $E_4$ consist of 4 points on both curves. Now we construct explicit $G$-equivariant acyclic theta characteristics on $C_1$ and $C_2$.

\begin{lemma}
There is an acyclic $G$-equivariant theta characteristic $\L$ on $C_1$.
\end{lemma}
\begin{proof}
Let $L = E_1 - E_3$ and $\L = \O(L)$. Looking at the quotients $C_1 / \la x \ra, C_1 / \la y, z \ra$ we find $E_1 \sim E_2, E_3 \sim E_4$. It follows that $\L$ is a theta characteristic. Consider the quotient $C_1 \to C_1 / \la z \ra$. From Lemma \ref{beauvillelemma} we have $h^0(L) = h^0(L') + h^1(L')$ where $L'$ is a divisor of degree 0 on the curve $C_1 / \la z \ra$ of genus 1. One checks that $L'$ is again a theta characteristic. From the quotient $C_1 / \la z \ra \to C_1 / \la x^2, z \ra$ we find $h^0(L) = 2h^0(L') = 4h^0(\P^1, \O(-1)) = 0$.
\end{proof}

\begin{lemma}
There is an acyclic $G$-equivariant line bundle $\N$ on $C_2$.
\end{lemma}
\begin{proof}
The curve $C_2$ is abstractly the same as $C_1$ (lies in the same family) but the action of $G$ on it is twisted by an automorphism. Namely, consider the automorphism $\phi: G \to G$ given by
\begin{align*} 
x & \mapsto xyz, \\ 
y & \mapsto y, \\
z & \mapsto x^2yz.
\end{align*}
Then the curve $C_2$ is one of the possible curves $C_1$ with the $G$-action given by $g \cdot x = \phi^{-1}(g)x$ where on the right hand side we consider the action on $C_1$. Thus if we let $\N = \O(E_1 - E_3)$ on $C_2$ then $\N$ has no regular sections and is equivariant.
\end{proof}

From the above lemmas we get the following theorem.

\begin{theorem}
Let $S=(C_1 \times C_2)/G$ be a surface isogenous to a higher product with $p_g=q=0$ and $G = (\Z/4 \times \Z/2) \rtimes \Z/2$. Then there are exceptional sequences of line bundles of maximal length 4 on $S$.
\end{theorem}

\section{Appendix: Explicit construction in the case $G = D_4 \times \Z_2$}

In this section we give explicit construction of the line bundle $\M$ on the curve $C_2$ in the case $G = D_4 \times \Z_2$ (see section \ref{sectiond4}). We compute obstructions to the existence of the equivariant structure on line bundles $\L$ and $\M$ and prove that they are inverse to each other, so the bundle $\L \boxtimes \M$ on the product $C_1 \times C_2$ is equivariant. The next lemma is elementary.

\begin{lemma} \label{hyperell}
Let $C$ be a hyperelliptic curve of genus $g$. The twofold covering $C \to \P^1$ has $2g+2$ ramification points. Let us label them arbitrarily as $x_1, \dots, x_{g+1}$, $y_1, \dots, y_{g+1}$. Then there is a rational function $f$ on $C$ which has a simple zero in each $x_i$, simple pole in each $y_i$ and for such $f$ we have $\sigma_{\ast}f = -f$ where $\sigma$ is the nontrivial automorphism of the covering $C \to \P^1$.
\end{lemma}

The group $G$ has a presentation 
$$
\langle x,y,z \mid x^4 = y^2 = z^2 = [x,z] = [y,z] = 1, x^y = x^{-1} \rangle,
$$
where $x^y = y^{-1}xy$, $[x,y] = xyx^{-1}y^{-1}$. We say that an element $g \in G$ is written in standard form if $g = x^k y^l z^m$, where $0 \le k \le 3$, $0 \le l,m \le 1$. Each element of $G$ has the unique standard form.

Now we compute the obstruction of $\L$.

\begin{lemma}
The line bundle $\L$ on $C_1$ has obstruction $$ \eta_{\L}(x^ky^lz^m, x^{k'}y^{l'}z^{m'}) = (-1)^{m(k' + l')} \cdot i^{-kl'} $$ for elements of $G$ in the standard form (here $i = \sqrt{-1}$).
\end{lemma}
\begin{proof}
Recall from section \ref{sectiond4} that $\L = \O(B_1 - B_2)$, where $B_1$ is one of $\la x^2, xy \ra$-orbits in $E_1$ and $B_2$ is any free $\la z \ra$-orbit. We will compute obstructions for $B_1$ and $B_2$.

Consider the covering $q: C_1 \to C_1' = C_1 / \la x^2, xy \ra$. The curve $C_1'$ has genus 0. The image of $E_1$ on $C_1'$ is two points $P, xP$. Let $B_1 = q^\ast P$. The divisor $B_1$ has degree 4. The $G$-orbit of divisor $B_1$ consists of 2 points: $B_1$ and $xB_1 = q^\ast(xP)$. Let $f'$ be a rational function on $C_1'$ with divisor $xP - P$ and $f = q^\ast f'$ be its pullback to $C_1$. The quotient map $C_1' \to C_1' / \la z \ra$ is ramified in 2 points $P, xP$. By Lemma \ref{hyperell} we have $z_\ast f' = -f'$. The function $f$ is $\la x^2, xy \ra$-invariant since it comes from $C_1'$. The function $f$ has divisor $xB_1 - B_1$ and $x_\ast f$ has inverse divisor $B_1 - xB_1$. Then for the right choice of a constant we have $x_\ast f = \frac 1f$, $y_\ast = \frac 1f$, $z_\ast = -f$. The divisor $B_1$ is invariant under the action of $\la x^2, xy, z \ra$ and we can put $\phi_g = 1$ if $g \in \la x^2, xy, z \ra$ and $\phi_g = f$ if $g \notin \la x^2, xy, z \ra$. Then we have 
$$ \eta_{B_1}(x^ky^lz^m, x^{k'}y^{l'}z^{m'}) = (-1)^{m(k' + l')}.$$
Consider the covering $q: C_1 \to C_1' = C_1 / \la z \ra$. The curve $C_1'$ has genus 0. The image of $E_2$ on $C_1'$ consists of 4 points, two of them have stabilizer $\la y \ra$, other two have $\langle x^2y \rangle$. Let $P$ be one of the points which is stabilized by $y$ and let $B_2 = q^\ast P$. It is a divisor of degree 2. There is a function $f'$ on $C_1'$ with divisor $xP - P$ and its pullback $f = q^\ast f'$ has divisor $xB_2 - B_2$. The functions $x^k_\ast f$ have divisors $x^{k+1}B_2 - x^kB_2$. The function $f \cdot x_\ast f \cdot \dotsc \cdot x^{k-1}_\ast f$ has divisor $x^k B_2 - B_2$. The function $f \cdot x_\ast f \cdot x^2_\ast f \cdot x^3_\ast f$ has trivial divisor since $x^4 = 1$. Thus multiplying $f$ by a constant we can assume that
\begin{equation}\label{fun}
f \cdot x_\ast f \cdot x^2_\ast f \cdot x^3_\ast f = 1.
\end{equation}
The divisor $B_2$ has an orbit of order 4 under the action of $G$, it consists of $B_2, \dots, x^3B_2$. We have $yB_2 = zB_2 = B_2$. Then for the divisor $B_2$ we can put $\phi_{x^ky^lz^m} = \phi_{x^k}$, where $\phi_{x^k} = f \cdot x_\ast f \cdot \dotsc \cdot x^{k-1}_\ast f$ (by \eqref{fun} it depends only on $k \mod 4$). The divisor of the function $f' \cdot x_\ast f'$ is $x^2P - P$, but $P$ and $x^2P$ are the only points ramification points of the morphism $C_1' \to C_1' / \la y \ra$, therefore by lemma \ref{hyperell} we have $y_\ast(f' \cdot x_\ast f') = - f' \cdot x_\ast f'$. The divisor of $y_\ast f$ is $y(xB_2 - B_2) = x^3B_2 - B_2$ and $x^{-1}_\ast f$ has divisor $B_2 - x^3B_2$, so we have 
$$
y_\ast f = \frac C {x^{-1}_\ast f}
$$
for some constant $C$.
And
$$
y_\ast(f \cdot x_\ast f) = y_\ast f \cdot x^{-1}_\ast y_\ast f = \frac C {x^2_\ast f} \cdot \frac C {x^3_\ast f} = C^2 f \cdot x_\ast f.
$$
Thus $C^2 = -1$. Changing $f$ by $if$ if needed we can assume that $C = i$. Note that such a change preserves \eqref{fun}. Now a simple calculation gives the cocycle for $B_2$: 
$$ \eta_{B_2}(x^ky^lz^m, x^{k'}y^{l'}z^{m'}) = i^{lk'}.$$
It remains to add two obstructions to finish the proof of the Lemma.
\end{proof}

Then we construct $\M$ which gives us an explicit construction of exceptional sequences.

\begin{lemma}
There is a $G$-invariant line bundle $\M$ of degree zero on $C_2$ with obstruction inverse to $\eta_L$.
\end{lemma}
\begin{proof}
We will construct divisors $A_1, A_2$ on $C_2$ with $G$-invariant classes in the Picard group with obstructions $\eta_{A_1} = (-1)^{mk'}$ and $\eta_{A_2} = i^{l(k' + 2m')}$. Consider the covering $q: C_2 \to C_2' = C_2 / H$, where $H = \la y, x^2y \ra$. In each of the fibers $E_1$, $E_3$ there are 4 points with stabilizer $\la y \ra$ and 4 points with stabilizer $\la x^2y \ra$. Thus the curve $C_2'$ has genus 1. The subgroup $H$ is normal in $G$ and there is an action of $G / H$ on $C_2'$. Each of $E_5$, $E_6$ consists of two free $H$-orbits and is mapped to two points on $C_2'$. Let us denote by $P$ a point in the image of $E_5$ and by $Q$ a point in the image of $E_6$. The image of $E_5$ on $C_2'$ is $\{P, xP\}$ because $P$ is stabilized by $H$ and $x^2z$, so $xP$ is a different point in the image of $E_5$. The image of $E_6$ is $\{Q, xQ\}$. Let the divisor $A_1$ on $C_2$ be equal to $q^{\ast}(P-Q)$. It has degree 0. 

We claim that the line bundle $\O(A_1)$ lies in the $Pic(C_2)^G$. We need to find isomorphisms $g_\ast \O(A_1) \to \O(A_1)$ for each $g \in G$, or in other words, rational functions $\phi_g$ which have divisors $gA_1 - A_1$. The divisor $A_1$ is invariant under the action of the subgroup $\la x^2, y, z \ra = \la H, x^2z \ra$. There are only two elements in the $G$-orbit of $A_1$: $A_1$ and $xA_1$. So we only need to find a function on $C_2$ with divisor $xA_1 - A_1$. Consider the covering $C_2' \to C_2'' = C_2 / \la x, x^2y, x^2z \ra$. Then $C_2''$ has genus 0 and $C_2' \to C_2''$ is the covering of degree two ramified in the points $P,xP,Q,xQ$.

By Lemma \ref{hyperell} there is a function $f'$ on $C_2'$ with divisor $P-xP-Q+xQ$. We denote its pullback to $C_2$ by $f$. The divisor of $f$ on $C_2$ is $xA_1 - A_1$. With the right choice of the  multiplicative constant $G$ acts on $f$ in the following way: $x_\ast f = \frac 1f$, $y_\ast f = f$, $z_\ast f = -f$. Now we can put $\phi_g = 1$ if $g \in \la x^2, y, z \ra$ and $\phi_g = f$ if $g \notin \la x^2, y, z \ra$. We see that $\O(A_1)$ lies in the invariant part of the Picard group. Computing the obstruction by the formula \eqref{obstr} we see that it is equal to 
$$ \eta_{A_1}(x^ky^lz^m, x^{k'}y^{l'}z^{m'}) = (-1)^{mk'}. $$

Now we will construct another divisor with $G$-invariant class on $C_2$. Since we will not need any more the details of the construction of $A_1$, we will reuse the same notation $P, Q, C_1'$ etc for new objects. Consider the covering $q: C_2 \to C_2' = C_2 / \langle x^2z \rangle$. The group $G / \langle x^2z \rangle$ acts on $C_2'$. Each of $E_1, E_3$ is mapped to 4 points on $C_2'$. In both cases two of the points are stabilized by $y$ on $C_2'$ and other two by $x^2y$. Let $P$ be one of the points in the image of $E_1$ stabilized by $y$, and $Q$ be such a point in the image of $E_3$. Then the 4 points in the image of $E_1$ are $P,xP,x^2P,x^3P$, the points $P$ and $x^2P$ are stabilized by $y$ and $xP,x^3P$ are exchanged by $y$ and stabilized by $x^2y$. The same is true about $E_3$ and points $x^iQ$. Let $A_2'$ be the divisor $P-Q$ on $C_2'$ and $A_2 = q^\ast A_2'$. Both divisors have degree 0. 
As in the case of $B_2$, the function $f' \cdot x_\ast f'$ on $C_2'$ has simple zeros or poles in all the points stabilized by $y$ and only in them, therefore $y_\ast (f \cdot x_\ast f) = - f \cdot x_\ast f$. Analogously we have 
$$y_\ast f = \frac C {x^{-1}_\ast f}$$
and we can assume $C = i$. The only difference is that now $\phi_{x^ky^lz^m} = \phi_{x^{k\pm 2m}}$, because we took quotient by $\langle x^2z \rangle$, not by $\langle z \rangle$. We get 
$$ \eta_{B_2}(x^ky^lz^m, x^{k'}y^{l'}z^{m'}) = i^{l(k' + 2m')}. $$
If we add the obstructions for line bundles $\L$ and $\M$ we will get
$$
\eta(x^ky^lz^m, x^{k'}y^{l'}z^{m'}) = (-1)^{mk'} \cdot i^{l(k'+2m')} \cdot (-1)^{mk' + ml'} \cdot i^{-lk'} = (-1)^{ml' + lm'}
$$
and this cocycle is cohomologous to zero because it is the differential of the 1-cochain $\beta:G \to \C^\times$ such that $\beta(x^ky^lz^m) = -1$ if both $k$ and $l$ are odd and $\beta = 1$ otherwise. 
\end{proof}

\end{document}